\documentclass{amsart}
\usepackage{amsmath,amssymb}
\usepackage{amsthm, thm-restate}
\usepackage{color,mathdots}
\usepackage{tikz}
\usepackage{tikz-cd}
\usepackage{enumitem}
\usepackage{hyperref}
\newtheorem{theorem}{Theorem}[section]
\newtheorem{lemma}[theorem]{Lemma}
\newtheorem{corollary}[theorem]{Corollary}
\newtheorem{fact}[theorem]{Fact}
\newtheorem{proposition}[theorem]{Proposition}

\newtheorem{theoremI}[theorem]{Theorem}

\theoremstyle{definition}
\newtheorem{example}[theorem]{Example}

\newtheorem{definition}[theorem]{Definition}
\newtheorem{assumption}[theorem]{Assumption}
\newtheorem{remark}[theorem]{Remark}

\newcommand{\D}{\mathcal{D}}
\newcommand{\C}{\mathcal{C}}
\newcommand{\N}{\mathbb{N}}
\newcommand{\Q}{\mathbb{Q}}

\newcommand{\m}{\mathfrak{m}}
\newcommand{\ideal}{\trianglelefteq}
\renewcommand{\epsilon}{\varepsilon}
\renewcommand{\delta}{\partial}
\DeclareMathOperator{\id}{id}

\DeclareMathOperator{\rk}{rank}
\DeclareMathOperator{\ord}{ord}

\title[A Basis Theorem for Rings with Commuting Operators]{A Basis Theorem for Rings with Commuting Operators in Characteristic Zero}
\author{Cas Burton}
\address{Cas Burton, Department of Mathematics, University of Manchester, Oxford Road, Manchester, United Kingdom M13 9PL}
\email{cas.burton@manchester.ac.uk}

\date{January 31, 2025}
\thanks{This work was supported by the Additional Funding Programme for Mathematical Sciences, delivered by EPSRC (EP/V521917/1) and the Heilbronn Institute for Mathematical Research.}

\begin{document}
\begin{abstract}
Motivated by the differential basis theorem of Kolchin and the difference-differential basis theorem of Cohn, in this paper we present a basis theorem for polynomial rings equipped with \textit{commuting} generalised Hasse-Schmidt operators (in the sense of Moosa and Scanlon \cite{MSGen}). We recover Kolchin and Cohn's results as special cases of our main theorem.
\end{abstract}

\maketitle

\tableofcontents

\section{Introduction}

Basis theorems, or rather ascending chain conditions on systems of ideals, are an important tool in commutative algebra. In 1890, Hilbert proved that every ideal of a polynomial ring in finitely many variables over a field has a finite generating set. Commonly known as Hilbert's Basis Theorem, this generalises to polynomial rings over Noetherian rings. Recall that a ring is Noetherian if it has the ascending chain condition on ideals, or equivalently, if every ideal has a finite generating set. In the context of differential algebra, Ritt and Raudenbush asked the natural question: can this be adapted to differential ideals in a differential polynomial ring over a differential field?

We briefly recall some key notions from differential algebra. A differential ring \((R,\delta)\) is a ring equipped with an additive map \(\delta\) that satisfies the Leibniz rule; \(\delta(rs) = r\delta(s) + \delta(r)s\) for \(r,s \in R.\) We call \(\delta\) a derivation and say that an ideal is differential if it is closed under \(\delta.\) The differential polynomial ring over \((R,\delta)\) with indeterminate \(x\) is denoted \(R\{x\}_{\delta}\) and is constructed in the natural way - it is the usual polynomial ring over \(R\) with indeterminates \(x\) and its formal derivatives \(\delta(x),\delta^2(x),\ldots\) . The natural differential analogue of Hilbert's Basis Theorem does not hold: in the differential polynomial ring \(R\{x\}_{\delta},\) the differential ideal generated by \(\delta(x)\delta^2(x), \delta^2(x) \delta^3(x), \ldots , \delta^n(x)\delta^{n+1}(x),\ldots\) is not finitely generated as a differential ideal (for details, see page 12 of \cite{ritt1932differential}). However, by considering a more restrictive class of ideals, an analogue of the basis theorem does hold. In 1934, Raudenbush \cite{raudenbushdifferential} established a basis theorem for \textit{radical} differential ideals in differential polynomial rings over a differential field of characteristic zero with a single derivation. Kolchin \cite{kolchinfrench} extended this result in 1961 to differential polynomial rings over differential rings of characteristic zero with multiple commuting derivations, provided that the base ring has the ascending chain condition on radical differential ideals. Note that if we are considering a differential ring with multiple derivations, the derivations must commute for such a basis theorem to hold. If they do not, it is clear that the ideal generated by \(\delta_2(\delta_1(x)), \delta_2(\delta_1^2((x)),\ldots\) is not finitely generated as a radical differential ideal.

Another area where basis theorems have been explored is that of difference algebra. A difference ring \((R, \sigma_1,\ldots,\sigma_n)\) is a ring equipped with injective endomorphisms \(\sigma_i\colon R\to R.\) Again, we require that the endomorphisms commute. In \cite{rittdifference}, Ritt and Raudenbush establish a difference basis theorem for perfect difference ideals. A difference ideal is an ideal closed under each endomorphism \(\sigma_i.\) A perfect difference ideal \(I\) satisfies the following additional condition: The inclusion \(\tau_1(a^{t_1})\cdots\tau_m(a^{t_m}) \in I\) (where \(\tau_j\) is a composition of the \(\sigma_i\)'s and \(t_j \in \N\)) implies \(a \in I.\) If a difference ring \(R\) (of arbitrary characteristic) has the ascending chain condition on perfect difference ideals, then so does the difference polynomial ring over \(R\) in finitely many variables. The difference polynomial ring is naturally constructed as in the differential case.

In 1970, Cohn \cite{cohn_1970} combined the above two basis theorems in his difference-differential basis theorem. For a ring \(R\) with both a difference and differential structure, where all operators commute, Cohn established that if \(R\) 
 is of characteristic zero and has the ascending chain condition on perfect difference-differential ideals, then so does the difference-differential polynomial ring in finitely many variables over \(R.\)\footnote{We note that Cohn presents the result in arbitrary characteristic. However, in \autoref{char-issues} we discuss issues that arise in positive characteristic.}

In \cite{MSGen}, Moosa and Scanlon generalised the notions of difference rings, differential rings and difference-differential rings with the introduction of rings with generalised Hasse-Schmidt operators (also called \(\D\)-rings). Thus it is natural to ask whether an analogue of the basis theorem holds in the wider context of \(\D\)-rings. Let \(K\) be a field and \(\D\) be a finite-dimensional \(K\)-algebra. A \(\D\)-ring \((R,e)\) is a \(K\)-algebra \(R\) equipped with a \(K\)-algebra homomorphism \(e \colon R \to R \otimes \D.\) Given a basis of \(\D,\) \(\epsilon_0,\ldots,\epsilon_m,\) we can write \(e\) coordinate-wise as \(\delta_0,\delta_1,\ldots,\delta_m;\) namely, \[e(-) =  \delta_0 (-) \otimes \epsilon_0 + \cdots + \delta_m(-) \otimes \epsilon_m.\] We investigate a subclass of these \(\D\)-rings, which we call \(\D^*\)-rings. These will be \(\D\)-rings where \(\delta_0,\ldots,\delta_m\) commute with each other. We will also assume that the so-called associated endomorphisms are injective. Note that this differs slightly from the assumptions and structures given in \cite{mtfo0char}, as Moosa and Scanlon instead require that at least one of the associated endomorphisms is the identity map of \(R\) (we make no such assumption). We will look at perfect \(\D\)-ideals - these are ideals closed under all of the operators \(\delta_0, \ldots,\delta_m\) such that the inclusion \(\tau_1(a^{t_1})\cdots\tau_m(a^{t_m}) \in I\) (where \(\tau_j\) is a composition of the associated endomorphisms and \(t_j \in \N\)) implies \(a \in I.\) 
The \(\D^*\)-polynomial ring with indeterminate \(x\) is denoted by \(R\{x\}_{\D^*}\) and is constructed analogously to the differential polynomimal ring. It is the usual polynomial ring with indeterminates \(x\) and its formal images under the operators \(\delta_0,\ldots,\delta_m\) and (some of) their compositions. (See \autoref{d-poly-rings} for more details.) Our main result, and the final result of this paper, is the following:

\begin{theoremI}[\(\D\)-Basis Theorem]
    Let \(R\) be a \(\Q\)-algebra and \((R,e)\) be a \(\D^*\)-ring. If \(R\) has the ascending chain condition on perfect \(\D\)-ideals, so does \(R\{x\}_{\D^*}.\)
\end{theoremI}

The paper is organised as follows. In \autoref{section-prelims}, we recall key notions related to \(\D\)-rings. In particular, we present a basis-free description of \(\D\)-rings with commuting operators, which we refer to as \(\D^*\)-rings and introduce the \(\D^*\)-polynomial ring over a \(\D^*\)-ring. We also give a brief introduction to conservative systems. In \autoref{section-reduction}, we introduce the notion of a ranked basis and a ranking on the \(\D^*\)-polynomial ring. We prove various technical lemmas, culminating with an important \(\D^*\)-reduction lemma (see \autoref{div_alg}). In \autoref{section-d-basis-theorem}, we discuss the notion of a characteristic set of \(\D\)-ideals of the \(\D^*\)-polynomial ring and then combine the results of previous sections to prove our main result; namely, the \(\D^*\)-basis theorem (see \autoref{d-basis-thm}).

\section{Preliminaries}\label{section-prelims}

Throughout this paper, we fix a field \(K\) and a finite-dimensional \(K\)-algebra \(\D\) (unless explicitly stated, the characteristic of \(K\) remains arbitrary).  In this section, we recall some basic notions of \(\D\)-rings. In particular, in \autoref{d-rings-comm}, we introduce the concept of \(\D\)-rings with commuting operators which we refer to as \(\D^*\)-rings. In \autoref{d-poly-rings}, we construct the \(\D^*\)-polynomial ring over a \(\D^*\)-ring \((R,e).\) Namely, the universal object in the category of \(\D^*\)-algebras over \((R,e).\) In \autoref{cons-sys}, we conduct a brief review of the theory of conservative systems and illustrate their role in proving the \(\D^*\)-basis theorem.

\subsection{\(\D\)-rings and \(\D^*\)-rings}\label{d-rings-comm}

Recall from \cite{MSGen} that a \(\D\)-ring is a pair \((R,e)\) consisting of a \(K\)-algebra \(R\) equipped with a \(K\)-algebra homomorphism \[e \colon R \to \D(R) := R \otimes_{K} \D.\]

\begin{remark}
    In \cite{mtfo0char}, \(\D\) is further equipped with a \(K\)-algebra homomorphism \(\pi \colon \D \to K\) and in the definition of a \(\D\)-ring it is required that \(e\) is a section of \(\pi.\) In this paper, we relax this extra condition.
\end{remark}

Given a linear \(K\)-basis \(\bar{\epsilon} = \{\epsilon_0, \epsilon_1, \ldots, \epsilon_n\}\) of \(\D ,\) \(e\) can be written as \[e(r) = \delta_0(r) \otimes \epsilon_0 + \delta_1(r) \otimes \epsilon_1 + \cdots + \delta_n (r) \otimes \epsilon_n\] where \(\delta_i \colon R \to R\) are \(K\)-linear operators. Note that \(\{1\otimes \epsilon_i \mid \epsilon_i \in \bar{\epsilon}\}\) is an \(R\)-basis of \(\D(R).\) Associating \(1 \otimes \epsilon_i\) with \(\epsilon_i,\) we will write \(``r \epsilon"\) for \(``r \otimes \epsilon".\) We say that \(\{\delta_0,\delta_1,\ldots,\delta_n\}\) are the coordinate maps of \(e\) with respect to \(\{\epsilon_0,\epsilon_1,\ldots,\epsilon_n\}.\)
As \(e\) is an algebra homomorphism,  the operators \(\delta_i\) satisfy a suitable product rule; namely, if the basis elements are related by \[\epsilon_j \epsilon_k = \sum_{i=0}^n \alpha_i^{j,k} \epsilon_i, \quad \text{with } \alpha_i^{j,k} \in K,\] then the following product rules holds for each \(i\) and \(r,s \in R:\)
\begin{equation}\label{product}
    \delta_i(rs) = \sum_{j,k=0}^n \alpha_i^{j,k} \delta_j(r)\delta_k(s).
\end{equation}

\begin{example}\label{Degs}
    \begin{enumerate}[label=(\alph*)]
        \item Let \(\D = K[\epsilon]/(\epsilon)^2\) with the usual \(K\)-algebra structure. Denote by \(\sigma = \delta_0\) and \(\delta = \delta_1\) the operators associated to the basis \(\{1,\epsilon\}.\) Then \((R,e)\) is a \(\D\)-ring if and only if \(\sigma\) is a \(K\)-endomorphism and \(\delta\) is a \(K\)-linear derivation on \(R\) twisted by \(\sigma.\) Namely, for all \(r,s \in R,\) \(\delta(rs) = \sigma(r)\delta(s) + \delta(r)\sigma(s).\) In the case that \(\sigma = \id_R,\) \(\delta\) is a derivation in the usual sense. Note that \(\sigma\) and \(\delta\) do not necessarily commute.
        \item Let \(\D = K^m\) with the product \(K\)-algebra structure. Denote by \(\sigma_1 = \delta_1, \ldots, \sigma_m = \delta_m\) the operators associated to the standard basis \(\{\epsilon_1, \ldots, \epsilon_m\}\) (i.e. \(\epsilon_i = (0,\ldots, 1,\ldots,0)\) with \(1\) in the \(i^{th}\) position). Then \((R,e)\) is a \(\D\)-ring if and only if \(\sigma_1, \ldots, \sigma_m\) are \(K\)-endomorphisms. In this case, \(R\) is a difference ring over \(K\) where the \(\sigma_i\)'s do not necessarily commute.
        \item Let \(\D = K[\nu_1,\ldots,\nu_n]/(\nu_1,\ldots,\nu_n)^2 \times K^m\) with the natural \(K\)-algebra structure. Denote by \(\sigma_0 = \delta_0,\) \(\delta_1,\ldots,\delta_n\) and \(\sigma_j = \delta_{n+j}\) the operators associated to the basis \(\{\epsilon_0,\epsilon_1,\ldots,\epsilon_n,\epsilon_{n+1},\ldots,\epsilon_{n+m}\}\) where \(\epsilon_0 = (1,0,\ldots,0),\) \(\epsilon_i = (\nu_i, 0, \ldots, 0)\) for \(1 \le i \le n,\) and \(\epsilon_{n+j} = (0,\ldots,1,\ldots,0)\) where the \(1\) occurs in the \((j+1)^{th}\) position. Then \((R,e)\) is a \(\D\)-ring if and only if \(\delta_1,\ldots,\delta_n\) are \(K\)-linear derivations twisted by \(\sigma_0\) and each \(\sigma_j\) is a \(R\)-endomorphism. If \(\sigma_0 = \id_R,\) then \(R\) is a difference-differential ring as seen in \cite{cohn_1970} but without the restriction that the operators commute.
        \item Let \(\D = K[\epsilon]/(\epsilon)^{n+1}\) with the usual \(K\)-algebra structure. Denote by \(\delta_0, \delta_1,\ldots,\delta_n\) the operators associated to the basis \(\{1,\epsilon,\epsilon^2,\ldots,\epsilon^n\}.\) Then \((R,e)\) is a \(\D\)-ring if and only if \(\delta_0\) is a \(R\)-endomorphism and for all \(r,s \in R,\) \(\delta_i(rs) = \sum_{j+k=i}\delta_j(r)\delta_k(s).\) In the case where \(\delta_0 = \id_R,\) \((\delta_1,\ldots,\delta_n)\) is a truncated Hasse-Schimdt derivation. Again, the \(\delta_i\)'s do not necessarily commute.
    \end{enumerate}
\end{example}

For further examples and standard constructions of \(\D\)-rings, see Section 3 of \cite{mtfo0char}. We now recall the basic notion of \(\D\)-ideals and \(\D\)-homomorphisms. We follow the presentation of \cite{mohamed2022weil}.

\begin{definition}[\(\D\)-ideal]
Let \((R,e)\) be a \(\D\)-ring. An ideal \(I \ideal R\) is said to be a \(\D\)-ideal if \(e(I) \subseteq I \otimes_{K} \D \ideal \D(R).\) Equivalently, for any basis \(\bar{\epsilon}\) of \(\D,\) if \(\{\delta_0, \ldots, \delta_n\}\) are the coordinate maps of \(e\) with respect to \(\bar{\epsilon},\) we have \(\delta_i(I) \subseteq I\) for each \(0 \le i \le n.\)
\end{definition}

Let \(S \subseteq R.\) We denote by \([S ]_{\D}\) the smallest \(\D\)-ideal of \(R\) containing \(S,\) and we call it the \(\D\)-ideal of \(R\) generated by \(S.\) This exists because intersections of \(\D\)-ideals are \(\D\)-ideals. One can readily check that \([ S]_{\D}\) can be described as the ideal generated by \(\{\delta_{i_1}\ldots\delta_{i_s}a \mid a \in S, s \in \N, 0 \leq i_j \leq m\},\) where \(\delta_0,\ldots,\delta_m\) are the operators associated to a basis \(\bar{\epsilon}.\)

We note that, as in \cite{mohamed2022weil}, \(\D\) can be regarded as a functor on the category of \(K\)-algebras and \(K\)-algebra homomorphisms. Namely, for any \(K\)-algebra homomorphism \(\phi \colon R \to S,\) we set \(\D(\phi) \colon \D (R) \to \D(S)\) to be \(\phi \otimes \id_{\D}.\) 

\begin{definition}
    Let \((R,e)\) and \((S,f)\) be \(\D\)-rings. We say \(\phi \colon (R,e) \to (S,f)\) is a \(\D\)-homomorphism if it is both a \(K\)-algebra homomorphism and the following diagram commutes:
\[
\begin{tikzcd}[row sep=25]
& \D (R) \arrow[r, "\D (\phi)"] & \D (S) \\
& R \arrow[u, "e"] \arrow[r, "\phi"] & S \arrow[u, "f"] \\
\end{tikzcd}
\] that is; for all \(r \in R,\) we have \((\phi \otimes \id_{\D})(e(r)) = f(\phi(r)).\) Equivalently, for any basis \(\bar{\epsilon}\) of \(\D,\) \(\phi \delta_i^R = \delta_i^S \phi.\) If \(S\) is an \(R\)-algebra, we call \((S,f)\) an \((R,e)\)-algebra if the structure map \(R \to S\) is a \(\D\)-homomorphism. We may also say that \((S,f)\) is a \(\D\)-algebra over \((R,e).\) If \((S,f)\) and \((T,g)\) are both \((R,e)\)-algebras and \(\phi \colon S \to T\) is a map between them, then we say that \(\phi\) is a \((R,e)\)-algebra homomorphism (or a \(\D\)-algebra homomorphism over \((R,e)\)) if it is an \(R\)-algebra homomorphism and a \(\D\)-homomorphism.
\end{definition}

So far, we have made no assumptions on the operators \(\delta_i,\) other than being additive and satisfying the product rule (\ref{product}). In the terminology of \cite{mtfo0char}, one may say that they are ``free". We will now restrict to a subclass of \(\D\)-rings; namely, those where the operators commute.

\begin{definition}[\(\D^*\)-ring]

Let \((R,e)\) be a \(\D\)-ring. We say that \((R,e)\) is a \(\D^*\)-ring, or a \(\D\)-ring with commuting operators, if the following diagram commutes
\begin{equation}\label{comm-diag}
\begin{tikzcd}[row sep=25]
& \D (R) \arrow[r, "\D (e)"] & \D(\D (R)) \arrow[dd, "\id_R \otimes \beta" inner sep=0.7mm]\\
R \arrow[ur, "e"] \arrow[dr, "e"'] && \\
& \D (R) \arrow[r, "\D(e)"'] & \D(\D(R)) \\
\end{tikzcd}
\end{equation}

where \(\beta \colon \D \otimes \D \to \D \otimes \D\) is the canonical isomorphism \(\beta(x \otimes y) = y \otimes x.\) Recall that \(\D(e) = e \otimes \id_{\D}.\)
\end{definition}

We justify our use of the terminology ``commuting operators" with the following lemma.

\begin{lemma}
Let \((R,e)\) be a \(\D\)-ring. Then \((R,e)\) is a \(\D^*\)-ring if and only if for any basis \(\bar{\epsilon}\) of \(\D\) (equivalently, there exists a basis \(\bar{\epsilon}\) of \(\D\) such that) the coordinate maps \(\{\delta_0, \delta_1, \ldots, \delta_m\}\) of \(e\) commute pairwise.
\end{lemma}

\begin{proof}
Let \(r \in R.\) Then with respect to a basis \(\bar{\epsilon},\) we have \(e(r) = \sum_{i=0}^m \delta_i(r) \otimes \epsilon_i.\) Thus we have:

\[\D(e)(e(r)) = \sum_{i,j=0}^m \delta_j(\delta_i(r)) \otimes \epsilon_j \otimes \epsilon_i.\]

The diagram (\ref{comm-diag}) commutes if and only if \(\D(e)(e(r)) = (\id_R \otimes \beta)(\D(e)(e(r))).\) This holds if and only if

\[\sum_{i,j=0}^m \delta_j(\delta_i(r)) \otimes \epsilon_j \otimes \epsilon_i = (\id_R \otimes \beta) \left (\sum_{i,j=0}^m \delta_j(\delta_i(r)) \otimes \epsilon_j \otimes \epsilon_i\right ) = \sum_{i,j=0}^m \delta_j(\delta_i(r)) \otimes \epsilon_i \otimes \epsilon_j.\]

Note that \(\{\epsilon_i \otimes \epsilon_j \mid 0 \leq i,j \leq m\}\) forms an \(R\)-basis of \(\D(\D(R)).\) By examining the coefficients of the tensor \(\epsilon_i \otimes \epsilon_j\) on both sides, we see that \(\delta_j(\delta_i(r)) = \delta_i(\delta_j(r));\) i.e. \(\delta_i\) and \(\delta_j\) commute.
\end{proof}

For examples of \(\D^*\)-rings, we may take any of the examples given in \autoref{Degs} with the additional restriction that the operators commute. In particular, from \autoref{Degs}(d), we obtain commuting truncated Hasse-Schmidt derivations (though these Hasse-Schmidt derivations are not necessarily iterative). We also recover differential rings as studied in \cite{kolchin73} from \autoref{Degs}(a), and from \autoref{Degs}(c) we recover difference-differential rings as seen in \cite{cohn_1970}.

\subsection{\(\D^*\)-polynomial rings}\label{d-poly-rings}

Let \((R,e)\) be a \(\D^*\)-ring (i.e. \((R,e)\) is a \(\D\)-ring with commuting operators). We now define the \(\D^*\)-polynomial ring over \((R,e).\) 

\begin{definition}[\(\D^*\)-polynomial ring]
Fix a basis of \(\D,\) \(\bar{\epsilon} = \{\epsilon_0,\epsilon_1,\ldots,\epsilon_m\},\) and let \(\{\delta_0,\delta_1,\ldots,\delta_m\}\) be the coordinate maps of \(e\) with respect to \(\bar{\epsilon}.\) Let \(I\) be any set and \(\bar{x} : = \{x_i \mid i \in I\}\) be a family of indeterminates. Let \[\bar{x}_{\D^*} = \{d^{\theta}x_i \mid i \in I, \theta \in \N^{m+1}\}\] be a new collection of (algebraically independent) indeterminates. We identify \(x_i\) with \(d^0x_i\) where \(0\) is the zero tuple in \(\N^{m+1}.\) The \(\D^*\)-polynomial ring over \((R,e)\) in indeterminates \(\bar{x}\) with respect to \(\bar{\epsilon}\) is the ring \[R\{\bar{x}\}_{\D^*}^{\bar{\epsilon}}:= R[\bar{x}_{\D^*}],\] equipped with the unique \(K\)-algebra homomorphism \(e' \colon R\{\bar{x}\}_{\D^*}^{\bar{\epsilon}} \to \D(R\{\bar{x}\}_{\D^*}^{\bar{\epsilon}})\) extending \(e \colon R \to \D(R)\) and satisfying
\[
d^{\theta}x_i\mapsto d^{\theta + 1_0}x_i \otimes \epsilon_0 + d^{\theta + 1_1}x_i \otimes \epsilon_1 + \cdots + d^{\theta + 1_m}x_i \otimes \epsilon_m
\]
where \(1_i \in \N^{m+1}\) contains a \(1\) in the \(i^{th}\) position (indexed from 0) as the only non-zero entry. When the context is clear, we will denote the \(\D\)-structure on \(R\{x\}_{\D^*}^{\bar{\epsilon}}\) by \(e\) (rather than \(e').\) As the generators are algebraically independent, it is clear that the operators commute. As such, \((R\{\bar{x}\}_{\D^*}^{\bar{\epsilon}},e)\) is a \(\D^*\)-algebra over \((R,e).\)
\end{definition}

We now observe that \(\D^*\)-polynomial rings are universal objects in the category of \(\D^*\)-rings. 

\begin{lemma}\label{univ-prop}
    Let \((R,e)\) be a \(\D^*\)-ring and \(\bar{\epsilon}\) a basis of \(\D.\) Suppose that \((S,f)\) is a \(\D^*\)-algebra over \((R,e)\) that is generated as a \(\D^*\)-ring by the tuple \(\bar{a}=(a_i)_{i \in I}\) over \((R,e).\) Let \(\bar{x} = (x_i)_{i \in I}\) be a tuple of indeterminates. Then there exists a unique, surjective \((R,e)\)-algebra homomorphism \(\phi \colon R \{\bar{x}\}_{\D^*}^{\bar{\epsilon}} \to S\) that maps \(x_i \mapsto a_i\) for each \(i \in I.\)
\end{lemma}
\begin{proof}
    For \(\theta = (\theta_0, \theta_1,\ldots,\theta_m) \in \N^{m+1},\) define an \(R\)-algebra homomorphism
\begin{align*}
    \phi \colon R\{\bar{x}\}_{\D^*}^{\bar{\epsilon}} &\to S \\
    d^{\theta}x_i &\mapsto \delta_0^{\theta_0} \circ \delta_1^{\theta_1} \circ \delta_2^{\theta_2} \circ \cdots \circ \delta_m^{\theta_m}( a_i ) \\
    r &\mapsto r, \quad \text{for } r \in R
\end{align*}

This is clearly a surjective \(R\)-algebra homomorphism. To be an \((R,e)\)-algebra homomorphism, we must check that \(\D (\phi) \circ e = f \circ \phi.\)

Let \(d^{\theta}x_i \in R\{\bar{x}\}_{\D^*}^{\bar{\epsilon}}.\) Then \((\D(\phi) \circ e )(d^{\theta}x_i)=\D(\phi)(\sum_j d^{\theta + 1_j} x_i \otimes \epsilon_j).\) It follows that
\begin{align*}
    (\D(\phi) \circ e )(d^{\theta}x_i) &= \D(\phi) \left ( \sum_j d^{\theta + 1_j} x_i \otimes \epsilon_j \right) \\
    &=(\phi \otimes \id_{\D})\left ( \sum_j d^{\theta + 1_j} x_i \otimes \epsilon_j \right) \\
    &= \sum_j \left (\delta_0^{\theta_0 + 1_0} \circ \delta_1^{\theta_1 + 1_j} \circ \delta_2^{\theta_2 + 1_j} \circ \cdots \circ \delta_m^{\theta_m + 1_j}\right )(a_i) \otimes \epsilon_j \\
    &= \sum_j \delta_j \circ \left(\delta_0^{\theta_0} \circ \delta_1^{\theta_1} \circ \delta_2^{\theta_2} \circ \cdots \circ \delta_m^{\theta_m}\right) (a_i) \otimes \epsilon_j \quad \left(\text{as }(S,f) \text{ is a }\D^*\text{-ring}\right)\\
    &= f \left (\delta_0^{\theta_0} \circ \delta_1^{\theta_1} \circ \delta_2^{\theta_2} \circ \cdots \circ \delta_m^{\theta_m} \right) (a_i) \\
    &= (f \circ \phi )(d^{\theta}x_i)
\end{align*}

Thus \(\phi\) is a \((R,e)\)-algebra homomorphism. Clearly, this is the unique \(\phi\) with the desired properties.
\end{proof}

\begin{corollary}
    Suppose that \(\bar{\epsilon}\) and \(\bar{\mu}\) are two bases of \(\D.\) Then \(R\{\bar{x}\}_{\D^*}^{\bar{\epsilon}}\) and \(R\{\bar{x}\}_{\D^*}^{\bar{\mu}}\) are isomorphic as \((R,e)\)-algebras.
\end{corollary}

\begin{proof}
    By \autoref{univ-prop}, the map \(\phi \colon R\{\bar{x}\}_{\D^*}^{\bar{\epsilon}} \to R\{\bar{x}\}_{\D^*}^{\bar{\mu}}\) that takes \(x_i\) to \(x_i\) is a surjective \(\D^*\)-algebra homomorphism over \((R,e).\) Since the family \(\bar{x}_{\D^*}\) is algebraically independent over \(R,\) \(\phi\) must be injective.
\end{proof}

In the remainder of this paper, we denote the unique \(\D^*\)-polynomial ring up to isomorphism as \(R\{\bar{x}\}_{\D^*}.\) We view the variable \(d^{\theta}x_i\) as \(\delta_0^{\theta_0} \circ \cdots \circ \delta_m^{\theta_m} (x_i)\) where \(\theta = (\theta_0, \theta_1,\ldots,\theta_m) \in \N^{m+1}.\)

\medskip
Using this construction, we recover the usual difference polynomial ring.

\begin{example}
     Let \(\D = K^m\) with the product \(K\)-algebra structure and \((R,e)\) be a \(\D^*\)-ring; that is, \((R,e)\) is a \(\D\)-ring as in \autoref{Degs}(b) with the additional condition that the endomorphisms pairwise commute. The \(\D^*\)-polynomial ring over \((R,e)\) coincides with the difference polynomial ring over \((R,\sigma_1,\ldots,\sigma_m).\)
\end{example}

\begin{remark}\label{diff-poly}
    We note that while we do not recover the differential polynomial ring as an instance of \(\D^*\)-polynomial rings, we do recover it as a homomorphic image. Let \(\D = K[\epsilon]/(\epsilon)^2\) with the usual \(K\)-algebra structure. Denote by \(\sigma = \delta_0\) and \(\delta = \delta_1\) the operators associated to the basis \(\{1,\epsilon\}.\) Let \((R,e)\) be a \(\D^*\)-ring such that \(\sigma\) is the identity on \(R;\) in other words, the structure of \((R,e)\) is just that of a differential ring with derivation \(\delta.\) Then the \(\D^*\)-polynomial ring over \((R,e)\) in the variable \(x,\) \(R\{x\}_{\D^*},\) is the usual polynomial ring over \(R\) in variables \(d^{\theta}x\) for \(\theta \in \N^2;\) in particular, this includes indeterminates of the form \(\sigma x, \sigma \delta x,\ldots \) . On the other hand, the differential polynomial ring over \(R\) in the variable \(x,\) \(R\{x\}_{\delta},\) is the usual polynomial ring over \(R\) in just the variables \(\delta^ix\) for \(i \in \N.\) It is then clear that \(R\{x\}_{\D^*}\) and \(R\{x\}_{\delta}\) are distinct objects. Nonetheless, as we will see in \autoref{consequence-differential}, one can still recover useful results as there is a surjective \((R,e)\)-algebra homomorphism from \(R\{x\}_{\D^*}\) to \(R\{x\}_{\delta};\) namely, the one that maps \(\sigma\delta^ix\) to \(\delta^ix.\)
\end{remark}

\subsection{Conservative systems of ideals}\label{cons-sys}

We briefly recall the notion of perfect conservative systems. Perfect conservative systems form a useful framework to carry out transfer of Noetherianity from a ring to an overring. In particular, using standard results on perfect conservative systems, one can show that if \(R\) has the ascending chain condition on radical ideals, then so does the polynomial ring \(R[x]\) (see Section 9 of Chapter 0 in \cite{kolchin73} for details).

\begin{definition}
Given a ring \(R,\) a conservative system, \(\C,\) is a set of ideals of \(R\) such that
\begin{enumerate}
    \item The intersection of any set of elements of \(\C\) is an element of \(\C;\)
    \item The union of any non-empty set of elements of \(\C,\) totally ordered by inclusion, is an element of \(\C.\)
\end{enumerate}
\end{definition}
A conservative system is  divisible if it contains \(I :s =\{x \in R \mid xs \in I\}\) for every \(I \in \C\) and \(s \in R.\) A conservative system is radical if every element is a radical ideal. We call a conservative system perfect if it is both divisible and radical.

Let \(\C\) be a conservative system of ideals of \(R\) and \(\Sigma \subset R.\) The \(\C\)-ideal generated by \(\Sigma,\) denoted \((\Sigma)_{\C},\) is the intersection of all elements of \(\C\) containing \(\Sigma;\) it is the smallest \(\C\)-ideal containing \(\Sigma.\)

\begin{definition}
    Let \(\C\) be a conservative system of ideals of \(R.\) We call \(\C\) Noetherian, or we say that \(R\) is \(\C\)-Noetherian, if any of the following equivalent conditions hold.
    \begin{itemize}
        \item Every element of \(\C\) is finitely generated as a \(\C\)-ideal;
        \item Every strictly increasing sequence of elements of \(\C\) is finite;
        \item Every nonempty set of elements of \(\C\) has a maximal element.
    \end{itemize}
\end{definition}

The following result for conservative systems can be found as Theorem 2.5 in \cite{poisson}.

\begin{fact}\label{poisson-thm}
    Let \(\C\) be a perfect conservative system of a ring \(R.\) Assume that for every prime \(\C\)-ideal, \(P,\) there exists a finite \(\Sigma \subset P\) and \(s \in R \setminus P\) such that \(P = (\Sigma)_{\C} : s.\) Then \(\C\) is Noetherian.
\end{fact}

We will use this fact when proving our main result, the \(\D^*\)-basis theorem, in \autoref{section-d-basis-theorem}.
\section{Reduction results}\label{section-reduction}

In this section we introduce two important assumptions (Assumptions \ref{ass-1} and \ref{ass-2}) that we will adhere to for the remainder of the paper. We discuss the notion of a ranked basis and use this to produce a simpler form of the product rule. Following this, we introduce the notion of a ranking of \(\D^*\)-polynomials. We define reduction of \(\D^*\)-polynomials and prove the main result for this chapter: the \(\D^*\)-reduction lemma (\autoref{div_alg}).

As \(\D\) is a finite-dimensional \(K\)-algebra, we can write \(\D\) as a finite product of local \(K\)-algebras \(\D = \D_1 \times \cdots \times \D_t.\) Throughout the rest of the paper, we make the following assumption.

\begin{assumption}\label{ass-1}
    For \(i = 1, \ldots, t,\) the residue field of \(\D_i\) is \(K.\)
\end{assumption}

We note that all the examples presented in \autoref{Degs} satisfy this assumption.

\begin{definition}[Ranked basis]\label{ranked-basis} Let \(\D_i\) be a local finite-dimensional \(K\)-algebra with maximal ideal \(\m_i\) and residue field \(K.\) 
\begin{enumerate}
    \item Let \(\bar{\epsilon}_i = \{\epsilon_0, \ldots, \epsilon_{m_i}\}\) be a basis for \(\D_i\) as a \(K\)-vector space with \(\epsilon_0 \in K^*\) and \(\epsilon_j \in \m_i\) for \(j = 1,\ldots, m_i.\) Define \(\nu_i(j)\) to be the smallest integer \(r\) such that \(\epsilon_j \in \m_i^r / \m_i^{r+1}\) (recall that by Nakayama's lemma \(\m_i\) is nilpotent). We say that \(\bar{\epsilon}_i\) is a \emph{ranked basis} for \(\D_i\) if \(\nu_i(j) \le \nu_i(k)\) for \(1 \leq j<k \leq m_i.\)
    \item For \(\D= \prod_{i=1}^t \D_i,\) we say that an ordered basis \(\bar{\epsilon}\) is a ranked basis for \(\D\) if \(\bar{\epsilon}\) is of the form \(\bar{\epsilon}_1 \cup \ldots \cup \bar{\epsilon}_t\) where each \(\bar{\epsilon}_i\) is a ranked basis of \(\D_i.\) Here we identify \(\D_i\) with its copy in \(\D.\)
\end{enumerate}
\end{definition}

Note that ranked bases exist. For \(\D_i,\) we can build a basis \(\bar{\epsilon}_i\) by concatenating bases for \(\m_i^j / \m_i^{j+1}\) for \(j = 0,1,\ldots\) .

\medskip
From now on we fix a ranked basis \(\bar{\epsilon}\) for \(\D.\) For a \(\D^*\)-ring \((R,e),\) we denote the coordinate maps of \(e\) with respect to \(\bar{\epsilon}\) by \[\Delta = \{\sigma_1, \delta_{1,1}, \ldots \delta_{1,m_1}, \ldots,\sigma_t, \delta_{t,1}, \ldots \delta_{t,m_t}\}.\] Note that \(\{\sigma_i, \delta_{i,1}, \ldots, \delta_{i,m_i}\}\) are the coordinate maps of \(pr_i \circ e \colon R \to \D_i(R)\) with respect to the basis \(\bar{\epsilon_i}\) where \(pr_i\) denotes the canonical projection \(\D \to \D_i.\) Additionally, note that each \(\sigma_i\) is an endomorphism of \(R\) and they correspond to the associated endomorphisms as appearing in Section 4 of \cite{mtfo0char}.

Just as difference rings are rings equipped with \emph{injective} endomorphisms, we assume that each of the associated endomorphisms are injective.

\begin{assumption}\label{ass-2}
    For \(i = 1, \ldots, t,\) the endomorphism \(\sigma_i\) is injective.
\end{assumption}

\begin{lemma} \label{prod_rule}
Let \(R\) be a \(\D^*\)-ring. The coordinate maps of \(e\) with respect to a ranked basis \(\bar{\epsilon}\) on \(R\) satisfy the following product rule for all \(r,s \in R,\) and for \(1 \le i \le t,\) \(1 \le j \le m_i:\)
\[\delta_{i,j}(rs) = \delta_{i,j} (r) \sigma_i(s) + \sigma_i(r) \delta_{i,j} (s) + \sum_{(p,q) \in \gamma_i (j)} \alpha_{i,j}^{p,q} \delta_{i,p} (r) \delta_{i,q} (s)\] where \(\gamma_i (j) := \{(p,q) \mid 1 \leq p, q \leq m_i, \nu_i(p) + \nu_i(q) \le \nu_i(j)\}\) and \(\alpha_{i,j}^{p,q}\) is the coefficient of \(\epsilon_{i,j}\) in the product \(\epsilon_{i,p} \cdot \epsilon_{i,q}\) in \(\D_i.\)
\end{lemma}

\begin{proof}
Note that for \(i \neq k,\) \(\epsilon_{i,p} \cdot \epsilon_{k,q} = 0,\) as each basis element is in a different local component. Then for all \(r,s \in R,\) by \((\ref{product})\) we get \[\delta_{i,j}(rs) = \delta_{i,j} (r) \sigma_i(s) + \sigma_i(r) \delta_{i,j} (s) + \sum_{1 \le p,q \le m_i} \alpha_{i,j}^{p,q}\delta_{i,p} (r) \delta_{i,q} (s)\] where \(\alpha_{i,j}^{p,q}\) is the coefficient of \(\epsilon_{i,j}\) in the product \(\epsilon_{i,p} \cdot \epsilon_{i,q}\) in \(\D_i.\) For a given \(p \geq 1,\) we have that \(\epsilon_{i,p} \in \m_i^{\nu_i(p)} / \m_i^{\nu_i(p)+1}.\) Thus \(\epsilon_{i.p} \cdot \epsilon_{i,q} \in \m^{\nu_i(p)+\nu_i(q)}\) when \(p,q \geq 1.\) As \(\epsilon_j \in \m_i^{\nu_i(j)} / \m_i^{\nu_i(j)+1},\) we have that \(\alpha_{i,j}^{p,q}\) is zero when \(\nu_i(p)+\nu_i(q) \ge \nu_i(j).\)

Thus we have \[\delta_{i,j}(rs) = \delta_{i,j} (r) \sigma_i(s) + \sigma_i(r) \delta_{i,j} (s) + \sum_{(p,q) \in \gamma_i (j)} \alpha_{i,j}^{p,q} \delta_{i,p} (r) \delta_{i,q} (s)\] where \(\gamma_i(j) := \{(p,q) \mid 1 \leq p,q \leq m_i, \nu_i(p) + \nu_i(q) \le \nu_i(j)\}.\)
\end{proof}

\begin{remark}
    Note that the standard bases provided in \autoref{Degs} are in fact ranked bases.
\end{remark}

Let \((R,e)\) be a \(\D^*\)-ring, \(\bar{x}=\{x_1, \ldots , x_n\}\) and \((R\{\bar{x}\}_{\D^*},e)\) the \(\D^*\)-polynomial ring over \((R,e)\) in variables \(\bar{x}.\) Let \(\bar{\epsilon}\) be a ranked basis for \(\D,\) and \(M: = |\bar{\epsilon} |\) (in other words, \(M = \dim_K \D\)).

For \(d^{\theta} x_j \in \bar{x}_{\D^*},\) we have \(\sigma_i (d^{\theta} x_j) = d^{\theta + 1_{i0}} x_j,\) and \(\delta_{i,p}(d^{\theta} x_j) = d^{\theta + 1_{ip}} x_j,\) where, for ease of notation, we denote by \(1_{ip}\) the \(M\)-tuple with the only non-zero entry a \(1\) in the position corresponding to \(\epsilon_{i,p}.\) We will sometimes denote \(\sigma_i\) by \(\delta_{i,0}.\)

For \(\theta \in \N^M\) and \(1 \leq i \leq t,\) let \(\theta_i'\) be the \(m_i\)-tuple choosing the entries of \(\theta\) corresponding to the \(\epsilon_{i,p}\) for \(1 \leq p \leq m_i.\) We denote by \(\ord_i(\theta)\) the sum of the entries of \(\theta_i'\) and call it the \(i^{th}\)-order of \(\theta.\) We define the order of \(\theta\) to be the sum of the \(i^{th}\)-orders of \(\theta\) and denote it by \(\ord^{\delta}(\theta).\) For \(u = d^{\theta}x_j \in \bar{x}_{\D^*},\) we define the \(i^{th}\)-order of \(u\) to be the \(i^{th}\)-order of \(\theta\) and the order of \(u\) to be the order of \(\theta.\) For example, the order of \(\sigma_i\delta_{i,1}\delta_{i,2}\) is two while the order of \(\sigma_i\) is zero.

\begin{definition}
     Recall that \(\Delta\) is the set of coordinate maps of \(e\) with respect to the ranked basis \(\bar{\epsilon}.\) A \emph{ranking} of \(\bar{x}_{\D^*}\) is a total ordering satisfying the additional conditions:
\begin{itemize}
    \item For all \(u \in \bar{x}_{\D^*}\) and \(\delta \in \Delta,\) \(u < \delta (u);\)
    \item For all \(u,v \in \bar{x}_{\D^*}\) and \(\delta  \in \Delta,\) \(u < v\) implies \(\delta (u) < \delta (v);\)
    \item For \(\delta_{i,j}, \delta_{i,k} \in \Delta\) and \(u \in \bar{x}_{\D^*},\) \(\nu_i(j) < \nu_i(k)\) implies \(\delta_{i,j}(u) < \delta_{i,k}(u)\) for any \(0 \leq j,k \leq m_i.\)
\end{itemize}
\end{definition}

If \(u,v \in \bar{x}_{\D^*},\) we say that \(v\) is a transform of \(u\) if \(v = \theta (u)\) for some composition of \(\Delta\)-operators \(\theta.\) If \(\ord ^{\delta}(\theta) >0,\) we say that \(v\) is a \(\delta\)-transform of \(u.\) If \(\ord^{\delta}(\theta)=0,\) we say that \(v\) is a \(\sigma\)-transform of \(u.\) Note that a transform of \(u\) is a \(\sigma\)-transform if and only if \(\theta\) consists only of compositions of the \(\sigma_i.\)

Any ranking is a well-order; that is, every non-empty subset has a least element. A ranking is sequential if it has order type \(\N;\) that is, every variable is of higher rank than only finitely many other variables. An example of a sequential ranking is obtained by ordering the set of variables \(d^{\theta}x_j\) lexicographically with respect to \((T,j, \theta_{M-1},\ldots,\theta_0)\) where \(T\) is the sum of all the entries of \(\theta.\)

From now on, we fix \(\bar{x}\) and a ranking on \(\bar{x}_{\D^*}.\) Let \(f \in R\{\bar{x}\}_{\D^*}\setminus R.\) We define the leader of \(f,\) \(u_f,\) to be the variable of highest rank appearing in \(f.\) We can write \(f\) in the following form: \[f = g_d \cdot u_f^d + \cdots + g_1 \cdot u_f + g_0\] where \(u_{g_i} < u_f\) and \(g_d\) is non-zero. We define the initial of \(f,\) denoted \(I_f,\) as \(g_d.\) For \(u \in \bar{x}_{\D^*},\) we define \(\deg_u(f)\) to be the highest power of \(u\) appearing in \(f,\) or to be \(0\) if \(u\) does not appear in \(f.\) We write \(\deg (f)\) for \(\deg_{u_f}(f).\) 
We can extend our ranking to a ranking on \(\D^*\)-polynomials. This is a pre-order. We say that \(\rk(g) < \rk(f)\) if \(u_g < u_f\) or if \(u_g=u_f\) and \(\deg(g) < \deg(f).\) If \(g\) and \(f\) have the same leader and degree, we say \(\rk(g)=\rk(f).\)

We define the separant of \(f,\) \(s_f,\) as the formal derivative of \(f\) with respect to \(u_f,\) that is \[s_f := \frac{\delta f}{\delta u_f} = \sum_{n=1}^d n g_n \cdot u_f^{n-1}.\] Note that \(\rk(I_f) < \rk(f)\) and \(\rk(s_f) < \rk(f).\)

In the following lemma, we establish a key fact about the rank of \(\delta_{i,j} (f) -\sigma_i (s_f) \delta_{i,j}(u_f)\) in comparison to \(\delta_{i,j}(u_f).\) We employ this fact multiple times when proving the \(\D^*\)-reduction lemma.

\begin{lemma}\label{separant}
Let \(f \in R\{\bar{x}\}_{\D^*} \setminus R\) and \(\delta_{i,j} \in \Delta' = \Delta \setminus \cup_i\{\sigma_i\}.\) Then \[\rk \big(\delta_{i,j} (f) -\sigma_i (s_f) \delta_{i,j}(u_f)\big) < \rk \big(\delta_{i,j} (u_f)\big).\]
\end{lemma}

\begin{proof}
Note that if \(f,g\) are two \(\D^*\)-polynomials over \(R,\) the leaders of \(f+g\) and \(f \cdot g\) are bounded above by the maximum of \(u_f,u_g\) with respect to the ranking.

Firstly, suppose that \(f\) is of the form \(g \cdot u_f^n\) for some \(n \in \N\) with \(u_g < u_f.\)

Let \(n=1.\) By the product rules for \(\Delta'\) in \autoref{prod_rule}, we have: \[\delta_{i,j}(g \cdot u_f) = \delta_{i,j} (g) \sigma_i(u_f) + \sigma_i(g) \delta_{i,j} (u_f) + \sum_{(p,q) \in \gamma_i (j)} \alpha_{i,j}^{p,q} \delta_{i,p} (g) \delta_{i,q} (u_f).\]

As our basis is ranked, for \((p,q) \in \gamma_i(j),\) we have \(p,q < j.\) By the definition of the ranking, we have that for any \(p<j,\) \(\delta_{i,p}(u_f) < \delta_{i,j}(u_f).\)  Let \(v\) be any variable appearing in \(g.\) As \(v < u_f,\) we must have \(\delta_{i,p} (v) < \delta_{i,j}(u_f)\) for any \(p < j.\) From this, we see that the leader of \(\delta_{i,j}(f)\) is bounded above by \(\delta_{i,j}(u_f)\) and we have \(\delta_{i,j}(f) = \sigma(g) \delta_{i,j}(u_f) + h\) with \(u_h < \delta_{i,j}(u_f).\) As (in this case) \(\sigma_i(g)=\sigma_i (s_f),\) we have that \[h = \delta_{i,j}(f) - \sigma_i(s_f) \delta_{i,j}(u_f)\] and \(\rk (h) < \rk (\delta_{i,j}(u_f)).\) Thus the lemma holds for polynomials of the form \(g \cdot u_f.\)

Suppose now that the result holds for all \(m<n.\) By assumption, we know that
\begin{equation}\label{ass-step}
\rk (\delta_{i,j}(g \cdot u_f^{n-1}) - (n-1)\sigma_i(g \cdot u_f^{n-2})\delta_{i,j}(u_f)) < \rk (\delta_{i,j}(u_f)).
\end{equation}
Note that as \(\sigma_i = \delta_{i,0} < \delta_{i,j},\) we have that \(\sigma_i(u_f) < \delta_{i,j}(u_f).\) It is easy to see that multiplying a polynomial by a variable smaller than its leader does not change its rank. Thus we have \(\rk (\delta_{i,j}(u_f)\sigma_i(u_f)) = \rk(\delta_{i,j}(u_f)).\) Multiplying the polynomials in (\ref{ass-step}) by \(\sigma_i(u_f),\) we obtain \[\rk \big(\delta_{i,j}(g \cdot u_f^{n-1})\sigma_i(u_f) - (n-1)\sigma_i(g \cdot u_f^{n-1})\delta_{i,j}(u_f)\big) < \rk \big(\delta_{i,j}(u_f)\sigma_i(u_f)\big).\]

By the product rules for \(\Delta',\) we have:
\begin{align*}
    \delta_{i,j}(g \cdot u_f^n) - n\sigma_i(g \cdot u_f^{n-1}) &= \delta_{i,j}(g \cdot u_f^{n-1}\cdot u_f) - n\sigma_i(g \cdot u_f^{n-1}) \\
    &= \delta_{i,j}(g \cdot u_f^{n-1})\sigma_i(u_f) - (n-1)\sigma_i(g \cdot u_f^{n-1})\delta_{i,j}(u_f) \, \\
    & \quad + \sum_{(p,q) \in \gamma_i (j)} \alpha_{i,j}^{p,q} \delta_{i,p} (g \cdot u_f^{n-1}) \delta_{i,q} (u_f)
\end{align*}
As in the base case, for \(v\) any variable appearing in \(g\) or \(u_f\) and for \(p<j,\) we have \(\delta_{i,p} (v) < \delta_{i,j}(u_f).\) Using this, we obtain
\begin{align*}
    \rk\Big(\delta_{i,j}(g \cdot u_f^n) - n\sigma_i(g \cdot u_f^{n-1})\Big) &= \rk \bigg(\delta_{i,j}(g \cdot u_f^{n-1})\sigma_i(u_f) - (n-1) \sigma_i(g \cdot u_f^{n-1})\delta_{i,j}(u_f) \\
    &\quad + \sum_{(p,q) \in \gamma_i (j)} \alpha_{i,j}^{p,q} \delta_{i,p} (g \cdot u_f^{n-1}) \delta_{i,q} (u_f)\bigg) \\
    &< \rk \Big(\delta_{i,j}(u_f)\sigma_i(u_f) + \sum_{(p,q) \in \gamma_i (j)} \alpha_{i,j}^{p,q} \delta_{i,p} (g \cdot u_f^{n-1}) \delta_{i,q} (u_f)\Big) \\
    &< \rk \Big(\delta_{i,j}(u_f)\sigma_i(u_f)\Big) \\
    &= \rk \big(\delta_{i,j}(u_f)\big)
\end{align*}
Thus the lemma holds for polynomials of the form \(g \cdot u_f^n.\)

We now consider an arbitrary \(\D^*\)-polynomial \(f\) with leader \(u_f.\) Recall that we can write \(f\) as \[f = g_d \cdot u_f^d + \cdots + g_1 u_f + g_0\] where \(u_{g_i} < u_f.\) As the operators \(\delta_{i,j}, \sigma_i\) are additive, we have that

\begin{align*}
    &\rk\Big(\delta_{i,j}(f) - \sigma_i(s_f)\delta_{i,j}(u_f)\Big) \\
    &=\rk\bigg(\delta_{i,j}\Big(\sum_{n=0}^d g_n \cdot u_f^n\Big)-\sigma_i\Big(\sum_{n=0}^d n g_n u_f^{n-1}\Big)\delta_{i,j}(u_f)\bigg) \\
    &=\rk\bigg(\sum_{n=0}^d \Big(\delta_{i,j}(g_n \cdot u_f^n) - n\sigma_i(g_n \cdot u_f^{n-1})\delta_{i,j}(u_f)\Big)\bigg) \\
    &< \rk\big(\delta_{i,j}(u_f)\big)
\end{align*}

where the inequality holds as the rank of each individual summand is less than \(\rk(\delta_{i,j}(u_f)),\) and the rank of a finite sum of \(\D^*\)-polynomials is bounded above by the rank of each individual summand. So the lemma holds for all \(\D^*\)-polynomials.
\end{proof}

\begin{remark}\label{auto-comp}
    Let \(f \in R\{\bar{x}\}_{\D^*}\) and \(\tau\) be a composition of \(\Delta\)-operators such that \(\ord^{\delta} (\tau) = 0.\) It is easy to see that
    \[\tau(f) = \tau(I_f) \tau(u_f)^d + h\] where \(d = \deg(f)\) and \(\rk (h) < \rk (u_f^d).\)
\end{remark}

\begin{remark} \label{composition}
    Recall that \(\Delta' = \Delta \setminus \cup_i \{\sigma_i\}.\) If \(\psi \neq \id\) is a composition of \(\Delta\)-operators, then \(\psi = \sigma_1^{n_1} \circ \cdots \circ \sigma_t^{n_t} \circ \phi\) for some natural numbers \(n_1,\ldots,n_t\) and \(\phi\) a composition of \(\Delta'\)-operators. Using \autoref{separant} repeatedly, we see that, for \(\phi \neq \id,\) \[\psi (f) = \left (\sigma_1^{\ord_1(\psi)+n_1} \circ \cdots \circ \sigma_t^{\ord_t(\psi)+n_t}\right )(s_f) \psi(u_f) + \tilde{h}\] such that \(\rk(\tilde{h}) < \rk(\psi(u_f)).\) We write \(\rho(\psi)\) for the operator \(\sigma_1^{\ord_1(\psi)+n_1} \circ \cdots \circ \sigma_t^{\ord_t(\psi)+n_t}.\) Note that \(\rk (\rho(\psi)(s_f)) <\rk (\psi(u_f)).\)
\end{remark}

To establish a division algorithm, we must first formalise what it means for one \(\D^*\)-polynomial to be reduced with respect to another.

\begin{definition}[Reduction]
Let \(f, g \in R\{\bar{x}\}_{\D^*}, f \notin R.\) We say \(g\) is reduced with respect to \(f\) if \emph{both} of the following hold
\begin{itemize}
    \item \(g\) is \emph{partially reduced} with respect to \(f;\) that is, \(g\) contains no \(\delta\)-transform of \(u_f;\)
    \item If \(v\) is a \(\sigma\)-transform of \(u_f\) that appears in \(g,\) then it appears with degree strictly less than \(\deg (f).\)
\end{itemize}
Note that if \(g \in R,\) then \(g\) is reduced with respect to any \(f \notin R.\) For \(A \subseteq R\{\bar{x}\}_{\D^*} \setminus R,\) we say that \(g\) is reduced with respect to \(A\) if \(g\) is reduced with respect to every element of \(A.\)
\end{definition}

Using this notion of reduction and our ranking of \(\D^*\)-polynomials, we now prove the \(\D^*\)-reduction lemma.

\begin{lemma}[\(\D^*\)-reduction lemma]\label{div_alg}
Let \(A \subseteq R\{\bar{x}\}_{\D^*} \setminus R.\) Then for any \(g \in R\{\bar{x}\}_{\D^*}\) there exist \(H,g_0 \in R\{\bar{x}\}_{\D^*}\) such that \(H\) is a product of \(\sigma\)-transforms of initials and separants of \(\D^*\)-polynomials in \(A,\) \(g_0\) is reduced with respect to \(A,\) \(\rk (g_0) \leq \rk (g),\) and \(H \cdot g \equiv g_0 \text{ mod} \, [A]_{\D}.\)    
\end{lemma}

\begin{proof}
    If \(g\) is reduced with respect to \(A,\) then we can take \(g_0 = g\) and \(H = 1.\) Therefore, we can assume that \(g\) is not reduced with respect to \(A.\) Let \(u_i\) denote the leader of \(a_i \in A,\) \(d_i\) the degree of \(a_i,\) \(s_i\) the separant of \(a_i\) and \(I_i\) the initial of \(a_i.\) Then as \(g\) is not reduced with respect to \(A,\) it contains some power \(\theta(u_i)^k\) of a transform of some \(u_i,\) where \(\theta\) is a \(\Delta\)-composition. If \(\ord(\theta)=0,\) then \(k \geq d_i.\) Such a term of highest possible rank is called the \(A\)-leader of \(g.\)
    
    Let \(\Sigma\) be the set of all \(\D^*\)-polynomials for which the lemma does not hold. Suppose that \(\Sigma \neq \emptyset\) and let \(g \in \Sigma\) be such that its \(A\)-leader \(v\) has the lowest possible rank and appears with lowest degree amongst all \(\D^*\)-polynomials in \(\Sigma\) with \(A\)-leader \(v.\) Then there are two possible situations; either \(v\) is a \(\delta\)-transform of some \(u_i,\) or \(v\) is a \(\sigma\)-transform of some \(u_i\) and appears with greater than or equal degree than \(d_i.\) In both cases, we can write \(g = g_1 \cdot v^k + g_2\) where \(g_1\) does not contain \(v\) and \(\deg_v(g_2) < k.\) Furthermore, we know that \(v = \theta (u_i)\) for some leader of \(a_i \in A.\)

    Assume that we are in the first case; that is, \(v = \theta(u_i)\) for some \(u_i\) and \(\ord (\theta) > 0.\) Consider the \(\D^*\)-polynomial \(r = \rho(\theta)(s_i) g - g_1 \cdot v^{k-1} \theta (a_i).\) As \(g \in \Sigma,\) \(r\) cannot be reduced with respect to \(A.\) By \autoref{composition}, we have that \(\theta(a_i) = \rho(\theta)(s_i)\theta(u_i) + h\) for some \(h\) with lower rank than \(\theta(u_i).\) Thus we have
    \begin{align*}
        r &= \rho(\theta)(s_i) g - g_1 \cdot v^{k-1} \theta (a_i) \\
        &= \rho(\theta)(s_i)\left(g_1 \cdot v^k + g_2\right) - g_1 \cdot v^{k-1} \left( \rho(\theta(s_i)\theta(u_i) + h\right) \\
        &= \rho(\theta)(s_i)g_2 - v^{k-1}hg_1
    \end{align*}
    Note that \(g_1\) does not contain \(v,\) \(\deg_v(g_2) <k\) and the rank of \(h\) and \(\rho(\theta)(s_i)\) are less than the rank of \(v,\) hence \(r\) is a \(\D^*\)-polynomial with \(A\)-leader of rank less than or equal to \(v\) and \(\deg_v(r) < k.\) Thus \(r \notin \Sigma\) and there exist \(\tilde{H}, \tilde{g} \in R\{\bar{x}\}_{\D^*}\) such that \(\tilde{H}\) is a product of \(\sigma\)-transforms of initials and separants of \(\D^*\)-polynomials in \(A,\) \(\tilde{g}\) is reduced with respect to \(A,\) \(\rk(\tilde{g}) \leq \rk(r),\) and \(\tilde{H} \cdot r \equiv \tilde{g} \text{ mod} \, [ A ]_{\D}.\) Thus we have that
    \begin{align*}
        \tilde{H} \cdot \rho(\theta)(s_i) \cdot g &= \tilde{H} \cdot r + \tilde{H} \cdot g_1 \cdot v^{k-1} \theta(a_i) \\
        &\equiv \tilde{H} \cdot r \text{ mod} \, [ A]_{\D} \\
        &\equiv \tilde{g} \text{ mod} \, [A ]_{\D}
    \end{align*}
    This is a contradiction, thus if \(g \in \Sigma,\) we must be in the second case; i.e. the \(A\)-leader of \(v\) is a \(\sigma\)-transform of some \(u_i.\)
    
    Let \(v = \tau(u_i)\) for some \(a_i.\) Consider the \(\D^*\)-polynomial \(r = \tau(I_i) g - g_1 \cdot v^{k-{d_i}} \tau(a_i).\) By \autoref{auto-comp}, we have that \(\tau(a_i) = \tau(I_i)v^{d_i} + h\) where \(h\) has lower rank than \(\tau(u_i)^{d_i}.\) Thus we have
    \begin{align*}
        r &= \tau(I_i) g - g_1 \cdot v^{k-{d_i}} \tau(a_i) \\
        &= \tau(I_i) \left(g_1 v^k +f_2\right) - g_1 \cdot v^{k-{d_i}} \left(\tau(I_i)v^{d_i}+h\right) \\
        &= \tau(I_i)g_2 - g_1 \cdot v^{k-{d_i}} h
    \end{align*}
    As \(g_1\) does not contain \(v,\) \(\deg_v(g_2)<k,\) the rank of \(\tau(I_i)\) is less than the rank of \(v,\) and the rank of \(h\) is less than the rank of \(\tau(u_i)^{d_i},\) we have that \(r\) is a \(\D^*\)-polynomial with \(A\)-leader of rank less than or equal to \(v\) and \(\deg_v(r) < k.\) Thus \(r \notin \Sigma\) and there exist \(\tilde{H}, \tilde{g} \in R\{\bar{x}\}_{\D^*}\) such that \(\tilde{H}\) is a product of \(\sigma\)-transforms of initials and separants of \(\D^*\)-polynomials in \(A,\) \(\tilde{g}\) is reduced with respect to \(A,\) \(\rk (\tilde{g}) \leq \rk (r),\) and \(\tilde{H} \cdot r \equiv \tilde{g} \text{ mod} \, [ A ]_{\D}.\) Thus we have that
    \begin{align*}
        \tilde{H} \cdot \tau(I_i) \cdot g &= \tilde{H} \cdot r + \tilde{H} \cdot g_1 \cdot v^{k-{d_i}} \tau(a_i) \\
        &\equiv \tilde{H} \cdot r \text{ mod} \, [ A ]_{\D} \\
        &\equiv \tilde{g} \text{ mod} \, [A ]_{\D}
    \end{align*}
    This is a contradiction and so \(\Sigma = \emptyset ,\) as desired.
\end{proof}
\section{\(\D^*\)-basis theorem}\label{section-d-basis-theorem}

In this section, we prove the main result; i.e. the basis theorem for \(\D^*\)-polynomial rings in characteristic zero. Towards this, we first discuss perfect \(\D\)-ideals and prove they form a perfect conservative system. We then introduce the notion of a characteristic set of a prime \(\D\)-ideal. We carry forward the notation and assumptions from the previous section; in particular, \autoref{ass-1} and \autoref{ass-2}.

\subsection{Perfect \(\D^*\)-ideals}
 \begin{definition}[Perfect \(\D\)-ideals]
    A \(\D\)-ideal of a \(\D^*\)-ring \(R\) is reflexive if for any \(a \in R\) and \(1 \le i \le t,\) \(a\sigma_i(a) \in I\) implies \(a \in I\) (recall that the \(\sigma_i\) are the associated endomorphisms).
    A reflexive \(\D\)-ideal which is also a radical ideal (i.e. \(\sqrt{I} = I\)) is called a perfect \(\D\)-ideal.
\end{definition}

\begin{lemma}\label{perfect-form-lemma}
    Let \(I\) be a \(\D\)-ideal of a \(\D\)-ring \((R,e)\) with associated endomorphisms \(\sigma_0,\sigma_1,\ldots,\sigma_t.\) Then \(I\) is a perfect \(\D\)-ideal if and only if the following condition holds:
    \begin{equation}\label{perfect-form}
        \tau_1(a)^{n_1} \cdots \tau_r(a)^{n_r} \in I \quad \implies \quad a \in I
    \end{equation}
    where \(n_j \in \N\) and \(\tau_j\) is a (possibly trivial) composition of \(\Delta\)-operators such that \(\ord^{\delta} (\tau_j) = 0.\)
\end{lemma}

\begin{proof}
    As \(a\sigma_i(a)\) and \(a^n\) are instances of \(\tau_1(a)^{n_1} \cdots \tau_r(a)^{n_r}, \) if \((\ref{perfect-form})\) holds for a \(\D\)-ideal \(I,\) then \(I\) is a perfect \(\D\)-ideal.
    
    If \(I\) is a perfect \(\D\)-ideal such that \(\tau_1(a)^{n_1} \cdots \tau_r(a)^{n_r} \in I,\) we can write \(\tau_1\) as \(\sigma_i \circ \tau_1'\) (after potentially reordering terms). Multiplying \(\tau_1(a)^{n_1} \cdots \tau_r(a)^{n_r}\) by \(\tau_1'(a)^{n_1} \sigma_i \left(\tau_2(a)^{n_2} \cdots \tau_r(a)^{n_r}\right),\) we obtain \(\tau_1'(a)^{n_1} \cdots \tau_r(a)^{n_r} \in I\) as \(I\) is a perfect \(\D\)-ideal. Repeating this process, we can remove instances of \(\sigma_i\) one at a time and remain within \(I.\) Thus we end up with \(a^n \in I\) for some \(n \in \N,\) and so \(a\in I\) as \(I\) is a radical ideal.
    
\end{proof}
\begin{remark}\label{perfect-ideal-corr}
    In \cite{levin08}, a perfect difference ideal is defined an ideal \(I\) closed under the operators \(\sigma_i\) such that \[\tau_1(a)^{n_1} \cdots \tau_r(a)^{n_r} \in I \quad \implies \quad a \in I\] where \(\tau_j\) is a (possibly trivial) composition of endomorphisms. In the case \(\D = K \times \cdots \times K\) and with the restriction that our operators commute and are injective, our definition of a perfect \(\D\)-ideal agrees with the notion of a perfect difference ideal by \autoref{perfect-form-lemma}.
\end{remark}
\begin{remark}\label{differential-ideals}
    In the case \(\D = K [\epsilon]/(\epsilon)^2\) and with the restriction that \(\delta_0 = \id_R\) (i.e. \(R\) is a differential ring), perfect \(\D\)-ideals correspond to radical differential ideals.
\end{remark}

Let \(S \subseteq R.\) We denote by \(\langle S \rangle_{\D}\) the smallest reflexive \(\D\)-ideal of \(R\) containing \(S.\) Similarly, we denote by \(\{S\}_{\D}\) the smallest perfect \(\D\)-ideal of \(R\) containing \(S.\) The ideal \(\{S\}_{\D}\) can be obtained from \(S\) via the following procedure (similar to the procedure called shuffling, found in \cite{levin08}). For any set \(M \subseteq R,\) let \(M'\) denote the set of all \(a \in R\) such that \(\tau_1(a)^{k_1}\cdots\tau_r(a)^{k_r} \in M\) for some \(\tau_1,\ldots, \tau_r\) compositions of \(\sigma_1,\ldots,\sigma_t\) and some \(k_1,\ldots,k_r \in \N.\) Let \(S_0 = S\) and inductively define \(S_{k+1} = [ S_k]_{\D}'.\) Clearly, \(S=S_0 \subseteq \{S\}_{\D}.\) The inclusion \(S_k \subseteq \{S\}_{\D}\) implies \([ S_k]_{\D} \subseteq \{S\}_{\D}\) and \(S_{k+1} = [ S_k]_{\D}' \subseteq \{S\}_{\D},\) since the \(\D\)-ideal \(\{S\}_{\D}\) is perfect. By induction \(S_k \subseteq \{S\}_{\D}\) for all \(k= 0,1,\ldots,\) hence \(\bigcup_{i=0}^{\infty} S_i \subseteq \{S\}_{\D}.\) By construction, we have that \(\bigcup_{i=0}^{\infty} S_i\) is a perfect \(\D\)-ideal of \(R,\) so it should contain \(\{S\}_{\D}.\) Thus \(\{S\}_{\D} = \bigcup_{i=0}^{\infty} S_i.\)

\medskip
We now use ranked bases (\autoref{ranked-basis}) to prove that the set of perfect \(\D\)-ideals is a perfect conservative system.

\begin{lemma}\label{rdcs}
Let \(R\) be a \(\D^*\)-ring with operators \(\Delta.\) Then the set of perfect \(\D\)-ideals of \(R\) forms a perfect conservative system.
\end{lemma}

\begin{proof}
Clearly, the intersection of any set of perfect \(\D\)-ideals is again a perfect \(\D\)-ideal, and the union of any non-empty set of perfect \(\D\)-ideals, totally ordered by inclusion, is again a perfect \(\D\)-ideal. It remains to show that the system is divisible; that is, for any perfect \(\D\)-ideal \(I,\) and any \(s \in R,\) \(I : s = \{x \in R \mid xs \in I\}\) is a perfect \(\D\)-ideal.

Recall that \(\Delta = \{\sigma_1, \delta_{1,1}, \ldots \delta_{1,m_1}, \ldots,\sigma_t, \delta_{t,1}, \ldots \delta_{t,m_t}\}\) and that for each \(1 \le i \le t, 1 \le j < k \le m_i,\) we have \(\nu_i(j) \le \nu_i(k).\) Recall that by \autoref{prod_rule}, the \(\Delta\)-operators satisfy the following product rule for all \(r,s \in R,\)
\[\delta_{i,j}(rs) = \delta_{i,j} (r) \sigma_i(s) + \sigma_i(r) \delta_{i,j} (s) + \sum_{(p,q) \in \gamma_i (j)} \alpha_{i,j}^{p,q} \delta_{i,p} (r) \delta_{i,q} (s)\] where \(\gamma_i (j) := \{(p,q) \mid 1 \leq p, q \leq m_i, \nu_i(p) + \nu_i(q) \le \nu_i(j)\}\) and \(\alpha_{i,j}^{p,q}\) is the coefficient of \(\epsilon_{i,j}\) in the product \(\epsilon_{i,p} \cdot \epsilon_{i,q}\) in \(\D_i.\)

Let \(I\) be a perfect \(\D\)-ideal of \(R\) and let \(s \in R.\) Let \(x \in I:s.\) Then as \(I\) is a \(\D\)-ideal, we have that \[\delta_{i,1}(xs) = \delta_{i,1}(x) \sigma_i(s) + \sigma_i(x) \delta_{i,1}(s) + \sum_{(p,q) \in \gamma_i(1)} \alpha_{i,1}^{p,q} \delta_{i,p}(x) \delta_{i,q}(s) \in I.\]

As \(\nu_i(1) \ge 1,\) we must have that \(\gamma_i(1) = \emptyset.\) Thus the above expression becomes \(\delta_{i,1}(x) \sigma_i(s) + \sigma_i(x) \delta_{i,1}(s) \in I.\) Multiplying this expression by \(s\sigma_i(\delta_{i,1}(x)),\) and as \(\sigma_i(x) \in I: s,\) we obtain \(\sigma_i(\delta_{i,1}(x))\delta_{i,1}(x)\sigma_i(s)s \in I.\) Thus as \(I\) is reflexive, \(\delta_{i,1}(x)s \in I,\) and so \(\delta_{i,1}(x) \in I:s.\)

Assume that for any \(x \in I:s,\) \(\delta_{i,p}(x) \in I:s\) for \(1 \le p \le m_i\) with \(\nu_i(p) < \nu_i(j).\)
Then we have that 
\[\delta_{i,j}(xs) = \delta_{i,j}(x) \sigma_i(s) + \sigma_i(x) \delta_{i,j}(s) + \sum_{(p,q) \in \gamma_i(j)} \alpha_{i,j}^{p,q} \delta_{i,p}(x) \delta_{i,q}(s) \in I.\]
As \((p,q) \in \gamma_i(j) = \{(p,q) \mid 1 \le p,q \le m_i, \nu_i(p) + \nu_i(q) \le \nu_i(j)\}\) and \(\nu_i(p) \ge 1,\) for every \(p \ge 1,\) we have that \(\nu_i(p) < \nu_i(j).\) Thus as \(\sigma_i(x) \in I:s,\) and \(\delta_{i,p}(x) \in I:s\) for every \(\delta_{i,p}\) appearing in the third term, we can multiply the expression by \(s\sigma_i(\delta_{i,j}(x))\) to obtain \(\sigma_i(\delta_{i,j}(x))\delta_{i,j}(x) \sigma_i(s) s \in I.\) As before, we obtain \(\delta_{i,j}(x) s \in I,\) and so \(\delta_{i,j}(x) \in I:s.\) By this inductive argument on \(\nu_i(j),\) we have shown \(I:s\) is a \(\D\)-ideal.

Since \(I\) is radical, \(I: s\) is radical. Suppose that \(x \cdot \sigma_i(x) \in I:s.\) Then \(x \sigma_i(x) s \in I,\) so \(x \sigma_i(x) s \sigma_i(s) \in I.\) Thus as \(I\) is reflexive, \(x s \in I,\) so \(x \in I:s,\) and \(I:s\) is a reflexive ideal.

Therefore \(I:s\) is a perfect \(\D\)-ideal and hence the set of all perfect \(\D\)-ideals of \(R\) form a perfect conservative system.
\end{proof}

\begin{remark}\label{poisson-applies}
Taking \(\C\) to be the conservative system consisting of perfect \(\D\)-ideals, we see that the ideal \(\{S\}_{\D}\) coincides with \((S)_{\C}.\) Thus we can rephrase the assumption in \autoref{poisson-thm} as follows: For every prime perfect \(\D\)-ideal, \(P,\) there exists a finite \(\Sigma \subset P\) and \(s \in R \setminus P\) such that \(P = \{\Sigma\}_{\D} : s.\)
\end{remark}

\subsection{Characteristic sets}

Recall that \((R,e)\) is a \(\D^*\)-ring, \(\bar{x} = \{x_1,\ldots,x_n\}\) and \((R\{\bar{x}\}_{\D^*},e)\) is the \(\D^*\)-polynomial ring over \((R,e)\) in variables \(\bar{x}.\) Additionally, recall that \(\bar{\epsilon}\) is a ranked basis for \(\D,\) with \(M = | \bar{\epsilon}|.\)

\begin{definition}[Autoreduced Set]
    Let \(A \subseteq R\{\bar{x}\}_{\D^*} \setminus R.\) We say that \(A\) is \textit{autoreduced} if for all \(f,g \in A,\) \(f\) is reduced with respect to \(g.\)
\end{definition}

\begin{proposition}\label{auto-red-prop}
    Let \(A\) be an autoreduced set. Then we have the following:
    \begin{itemize}
        \item For any \(f,g \in A,\) \(u_f \neq u_g;\)
        \item \(A\) is a finite set;
        \item \(A\) can be written as \(\{a_1,a_2, \ldots, a_n\}\) where \(\rk (a_i) < \rk (a_j)\) for \(i < j.\)
    \end{itemize}
\end{proposition}

\begin{proof}
    Let \(f,g \in A.\) If \(f\) and \(g\) have the same leader, it must appear in \(f\) with degree strictly less than \(\deg(g)\) and in \(g\) with degree strictly less than \(\deg (f).\) This is impossible. Thus for any \(f,g \in A,\) \(u_f \neq u_g.\)

    Let \(V\) be the set of variables that occur as the leader of some \(a \in A.\) We can partition \(V\) into finitely many sets \(V_i\) such that \(v \in V_i\) if and only if \(v = d^{\theta}x_i\) for some tuple of natural numbers \(\theta.\)

    Let \(\Theta_i = \{\theta \mid d^{\theta}x_i \in V_i\},\) and \(\Theta_i^*\) be the set of minimal elements of \(\theta\) under the product ordering on \(\N^M.\) By Dickson's Lemma \cite{Figueira_2011}, we have that \(\Theta_i^*\) is a finite set. Define \(W_{\theta} = \{\theta' \in \Theta_i \mid \theta < \theta' \}.\) Clearly, \(\Theta_i = \bigcup_{\theta \in \Theta_i^*} W_{\theta}.\)

    Suppose that \(d^{\theta}x_i = u_a\) and \(\deg (a) = n.\) Then if \(\theta' \in W_{\theta},\) the corresponding variable \(d^{\theta'}x_i\) can appear with degree at most \(n-1\) in any element of \(A\) by definition of reduction. Thus we can partition \(W_{\theta}\) into a finite number of sets \(X_j = \{\theta' \in W_{\theta} \mid d^{\theta'}x_i = u_b, b \in A, \deg(b) = j\}.\) It is clear to see by the definition of reduction that \(X_{n-1}\) is a subset of the minimal elements of \(W_{\theta'}.\) Thus again by Dickson's Lemma, \(X_{n-1}\) is finite. For any \(1\le j \le n-2,\) we see that \[X_j \subseteq \min (W_{\theta'}) \cup \bigcup_{\substack{\psi \in X_k \\ k>j}} \min(W_{\psi}).\] As \(X_{n-1}\) is finite, using Dickson's Lemma repeatedly, we see that \(X_j\) is finite for all \(1 \le k \le n-1.\)

    We have now shown that for each \(\theta,\) \(W_{\theta}\) is a finite union of finite sets and is thus finite. Thus for each \(i,\) \(\Theta_i\) is finite. As the leaders of distinct polynomials in \(A\) are distinct, \(\Theta_i\) is in bijection with \(V_i.\) Thus \(V\) is the finite union of the \(V_i\)'s, which are finite, and so is finite.
     
    As the leaders of elements of \(A\) are distinct, the ranking is total on \(A.\) As \(A\) is finite, we can write \(A\) as \(\{a_1, a_2, \ldots, a_n\}\) where \(\rk (a_i) < \rk (a_j)\) for \(i <j.\)
\end{proof}

\begin{definition}[Pre-order on Autoreduced Sets]
Given two autoreduced sets \(A=\{a_1 , \ldots , a_k\}\) and \(B=\{b_1 , \ldots,b_l\},\) we say that \(A \prec B\) if both:
\begin{itemize}
    \item for some \(i \le k,\) \(\rk (a_i) < \rk (b_i)\) and \(\rk(a_j) = \rk(b_j)\) for \(j <i;\)
    \item \(l<k\) and for all \(i \le l,\) \(\rk(a_i)= \rk( b_i).\)
\end{itemize}
This defines a pre-order on autoreduced sets. 
\end{definition}

\begin{lemma}
Let \(\mathcal{B}\) be a non-empty set of autoreduced sets in \(R\{\bar{x}\}_{\D^*}.\) Then \(\mathcal{B}\) has a minimal element with respect to the above pre-order.
\end{lemma}

\begin{proof}
    This can be proved as in Levin \cite{levin08} but we provide details. Let \(\mathcal{B}_1\) be the set of autoreduced sets \(\mathcal{A} \in \mathcal{B}\) such that \(|\mathcal{A}| \ge 1\) whose first element is of the lowest possible rank. Inductively define \(\mathcal{B}_{n+1}\) by taking the elements of \(\mathcal{A} \in \mathcal{B}_n\) with \(|\mathcal{A}| \ge n\) whose \((n+1)^{th}\) element is of the lowest possible rank. Let \(v_i\) be the leader of the \(i^{th}\) polynomial of any autoreduced set in \(\mathcal{B}_i.\) If every \(\mathcal{B}_i\) were non-empty, then we would have an infinite sequence of variables that are reduced with respect to one another. As autoreduced sets are finite, there is some \(\mathcal{B}_{n+1} = \emptyset.\) Any autoreduced set in \(\mathcal{B}_n\) is minimal with respect to this pre-order. Thus \(\mathcal{B}\) contains a minimal element.
\end{proof}

\begin{definition}[Characteristic Sets]\label{char-def}
    Let \(I\) be an ideal of \(R\{\bar{x}\}_{\D^*}.\) Let \(\mathcal{B}_I\) be the collection of autoreduced sets \(A \subset I\) such that for every \(f \in A,\) \(s_f \notin I.\) We say that a minimal element of \(\mathcal{B}_I\) is a characteristic set of \(I.\) Note that as the empty set is an autoreduced set contained in \(I,\) \(\mathcal{B}_I\) is always non-empty.
\end{definition}

Given any ideal \(I\) of \(R\{\bar{x}\}_{\D^*},\) we can construct a characteristic set \(A\) of \(I\) by the following procedure. Let \(B = \{f \in I \setminus R \mid s_f \notin I\}.\) If  \(B = \emptyset,\) then we may take \(A = \emptyset.\) Otherwise, take \(f_1\) to be any polynomial of minimal rank in \(B.\) Let \(A_1 =\{f_1\}.\)

For \(n>0,\) define recursively \(B_n = \{f \in B \mid f\) is reduced with respect to \(A_n,\) \(s_f \notin I\}.\) If \(B_n = \emptyset,\) take \(A=A_n.\) Take \(f_{n+1}\) to be any polynomial of minimal rank and degree in \(B_n.\) Let \(A_{n+1} = A_n \cup \{f_{n+1}\}.\) From \autoref{auto-red-prop}, we know that any autoreduced set is finite. Thus this process must terminate after finitely many steps, and we have constructed \(A = \bigcup A_n\) as an autoreduced set of \(I.\)

Note that by construction, \(A = \{f_1, \ldots, f_n\}\) is an autoreduced set with \(\rk(f_i) < \rk(f_j)\) for all \(i <j.\) At each step, \(f_i\) is an element of smallest rank reduced with respect to \(A_{i-1}\) with \(s_{f_i} \notin I.\) Thus any distinct autoreduced set containing the same number of elements cannot be lower than \(A\) in the pre-order on autoreduced sets. Similarly, as the process has terminated, we have added as many elements as possible, and there is no autoreduced set lower than \(A\) in the pre-order. Therefore \(A\) is a characteristic set of \(I.\)

\begin{lemma}
    Let \(I\) be an ideal of \(R\{\bar{x}\}_{\D^*}\) with characteristic set \(\C.\) If \(c \in \C,\) then \(I_c \notin I.\)
\end{lemma}

\begin{proof}
    Let \(c \in \C\) and \(d\) be the degree of \(u_c\) in \(c.\) The polynomial \(c'=c-I_i u_c^d\) must either be free of \(u_c\) or \(u_{c'} = u_c\) and \(u_c\) appears with degree strictly less than \(d.\) Thus \(\rk(c') < \rk(c).\) If \(I_i \in I,\) then \(c' \in I.\) Then the set \[\C' = \{c_i \mid \rk(c_i) < \rk(c)\} \cup \{c'\}\] is an autoreduced set for \(I\) and \(\C' \prec \C\) contradicting minimality of \(\C.\)
\end{proof}

\begin{lemma} \label{reduced_C}
Let \(I\) be an ideal of \(R\{\bar{x}\}_{\D^*}\) with characteristic set \[\C=\{c_1, \ldots , c_m\}.\] If \(f \in I\) is reduced with respect to \(\C\) and \(s_f \notin I,\) then \(f \in (I\cap R)\) where \((I \cap R)\) is the ideal of \(R\{\bar{x}\}_{\D^*}\) generated by \(I \cap R \trianglelefteq R.\)
\end{lemma}
\begin{proof}
Let \(f \in I\) be reduced with respect to \(\C\) with \(s_f \notin I.\) Towards a contradiction, suppose \(f \notin (I \cap R).\) Either \(\rk(c_i) < \rk(f)\) for every \(i,\) or for some \(1 \le j \le m,\) we have that \(\rk(f) < \rk(c_j).\) Note that \(\rk(f) \neq \rk(c_i)\) for each \(i,\) as \(f\) is reduced with respect to each \(c_i\) so they cannot have the same leading term and degree.

If \(\rk(c_i) < \rk(f)\) for every \(c_i \in \C,\) then the set \[\C'=\{c_1,\cdots , c_m , f\}\] is an autoreduced set for \(I.\) As it has the same initial \(m\) elements, but a larger cardinality, \(\C'\) is smaller than \(\C.\)
Otherwise, there is some smallest \(c_j \in \C\) such that \(\rk(f)< \rk(c_j).\) For \(j > 1,\) we set \[\C' = \{ c_1 , \cdots , c_{j-1} , f\}\] and for \(j=1,\) we set \(\C' = \{f\}.\) This is an autoreduced set for \(I\) and \(\C' \prec \C.\)
In either case, this contradicts minimality of \(\C\) and thus \(f\) must be contained in \(R\) if it is reduced with respect to \(\C.\)

\end{proof}

\begin{corollary}\label{reduced_corollary}
    Assume \(R\) is a \(\Q\)-algebra. Let \(I\) be an ideal of \(R\{\bar{x}\}_{\D^*}\) with characteristic set \(\C.\) If \(f \in I\) is reduced with respect to \(\C,\) then \(f \in (I \cap R).\)
\end{corollary}

\begin{proof}
    Assume \(f\) is a counterexample of minimal rank. By \autoref{reduced_C}, \(s_f \in I.\) Since \(s_f\) is reduced with respect to \(\C\) and \(\rk (s_f) < \rk (f),\) by minimality of \(\rk (f),\) we have \(s_f \in (I \cap R).\) Write \[f= g_0 + g_1u_f + \cdots + g_d u_f^d.\] Then \[s_f = g_1 + 2g_2 u_f + \cdots + d g_d u_f^{d-1}.\]
    Since \(s_f \in (I \cap R),\) we can write \(s_f = \alpha_1 h_1 + \cdots + \alpha_m h_m\) with \(\alpha_i \in I \cap R\) and \(h_i \in R\{\bar{x}\}_{\D^*}.\) Rearranging this as a polynomial in \(u_f,\) we get \(g_1, 2g_2, \ldots, dg_d \in (I \cap R).\) As \(R\) is a \(\Q\)-algebra, we get \(g_1,\ldots, g_d \in (I \cap R).\) Thus \(g_1 u_f + \cdots g_d u_f^d \in (I\cap R).\) Additionally, since \(g_0\) is reduced with respect to \(\C\) and \(\rk (g_0) < \rk(f),\) minimality of \(\rk (f)\) implies that \(g_0 \in (I \cap R).\) Hence \[f = g_0 + g_1 u_f + \cdots + g_d u_f^d \in (I \cap R),\] a contradiction. 
\end{proof}
    
\begin{remark}\label{char-issues}
    We now discuss some issues arising in positive characteristic. In positive characteristic, the above corollary does not hold in general. For instance, let \((K,\delta)\) be a differential field with one derivation. Take \(I\) to be the ideal of \(R\{x\}_{\D^*} = R\{\sigma^i \delta^i x\}\) generated by \(f(x) = x^p.\) Then the empty set is a characteristic set of \(I\) (since \(s_f = 0 \in I\)). Thus \(f\) is reduced with respect to \(\emptyset\) but \(f \notin (I \cap K) = (0).\)
    
    In \cite{cohn_1970}, Cohn presents an argument for the difference-differential basis theorem in arbitrary characteristic. To discuss issues arising there, we will briefly present some notions from Cohn's work. We are now working in the difference-differential polynomial ring in one variable \(x.\) Take an ideal \(I\) with characteristic set \(\C.\) For a polynomial \(c \in \C,\) let \(w_c = \psi(x)\) be the unique variable such that \(u_c\) is a \(\sigma\)-transform of \(w_c\) and \(\psi\) is a composition of derivations. Let \(V\) be the set of variables \(\theta(x)\) such that \(\theta\) is a composition of derivations and no \(v \in V\) is a proper derivation of any \(w_c\) for \(c \in \C.\) Cohn states without justification that if \(f \in I\) is reduced with respect to \(\C,\) then \(f \in (I \cap R)[V^*]\) where \(V^*\) is the set of \(\sigma\)-transforms of the set \(V.\) In positive characteristic, it is not clear to us why this holds and thus it is not clear how to adapt Cohn's argument in positive characteristic to our setup. As such, we restrict ourselves to characteristic zero and leave these issues to be addressed in future work.\footnote{In Cohn's notation, this statement is presented as ``Every member of \(P\) reduced with respect to \(F\) is in \(R_0[V^*]\).". It can be found in the proof of Theorem III (Section 5) on page 1232 of \cite{cohn_1970}.}
\end{remark}

\subsection{A \(\D^*\)-basis theorem}
We now prove our main result.

\medskip
For a subset \(\Sigma \subseteq R,\) \([\Sigma]_{\D},\) \(\langle \Sigma \rangle_{\D},\) and \(\{ \Sigma \}_{\D}\) denote the \(\D\)-ideal, the reflexive \(\D\)-ideal and the perfect \(\D\)-ideal generated by \(\Sigma,\) respectively. For an ideal \(I\) of a ring \(R\) and \(s \in R,\) we write \(I \colon s^{\infty} = \{r \in R \mid rs^n \in I \text{ for some } n \in \N\}.\) 

\begin{lemma}\label{ideal-form}
    Assume \(R\) is a \(\Q\)-algebra. Let \(P\) be a prime perfect \(\D\)-ideal of \(R\{\bar{x}\}_{\D^*}\) and \(\C\) be a characteristic set of \(P.\) Let \(H = \prod_{c\in\C} I_c s_c\) and \(B = P \cap R.\) Then \(H \notin P\) and \(P = \langle \C \cup B \rangle_{\D} \colon H^{\infty}.\)
\end{lemma}

\begin{proof}
    By \autoref{char-def}, \(s_c\) is not in \(P\) and by \autoref{reduced_C},  \(I_c\) is not in \(P.\) As \(P\) is prime, \(H \notin P.\) It is clear that \(\C \cup B\) is a subset of \(P.\) Using again that \(P\) is prime and \(H \notin P,\) we have \(\langle \C \cup B \rangle_{\D} \colon H^{\infty} \subseteq P.\) It remains to show that \(P \subseteq \langle \C \cup B \rangle_{\D} \colon H^{\infty}.\)
    
    Let \(f \in P.\) By \autoref{div_alg}, we can find \(f_0\) reduced with respect to \(\C\) such that \[\tilde{H} f \equiv f_0 \text{ mod} [\C ]_{\D}\] where \(\tilde{H}\) is a product of \(\sigma\)-transforms of \(I_c\) and \(s_c.\) As \(f_0\) is reduced with respect to \(\C,\) by \autoref{reduced_corollary}, \(f_0 \in (P \cap R) = (B).\) Thus we have that \[\tilde{H} f \in [\C \cup B]_{\D}.\]

    We now pass to the reflexive ideal generated by \(C \cup B.\) Multiplying \(\tilde{H} f\) by the appropriate \(\sigma\)-transforms of \(f, I_c\) and \(s_c\) and using the reflexiveness of \(\langle \C\cup B \rangle_{\D},\) we obtain \(H^mf \in \langle \C\cup B \rangle_{\D}\) for some \(m \in \N.\) Thus \(P = \langle \C \cup B \rangle_{\D} \colon H^{\infty}.\)
\end{proof}

\begin{remark}
    In the notation of the above lemma, if \(R\) is a field, then \(B = (0).\) Thus \(P = \langle \C \rangle_{\D} \colon H^{\infty}.\)
\end{remark}

\begin{theorem}[\(\D\)-Basis Theorem]\label{d-basis-thm}
    Let \((R,e)\) be a \(\D^*\)-ring with \(R\) a \(\Q\)-algebra and \(\bar{x} = (x_1,\ldots,x_n).\) If \(R\) has the ascending chain condition on perfect \(\D\)-ideals, then so does \(R\{\bar{x}\}_{\D^*}.\)
\end{theorem}

\begin{proof}
By \autoref{rdcs}, the collection of perfect \(\D\)-ideals of \(R\) form a perfect conservative system of ideals. By \autoref{poisson-thm} and \autoref{poisson-applies}, it is enough to show if \(P\) is a prime perfect \(\D\)-ideal, then there is a finite \(\Sigma \subset P\) and \(s \in R\{\bar{x}\}_{\D^*} \setminus P\) such that \(P = \{\Sigma\}_{\D} \colon s.\)

Let \(P\) be a prime perfect ideal with characteristic set \(\C.\) By \autoref{ideal-form}, \(P\) can be written as \(\langle \C \cup B \rangle_{\D} \colon H^{\infty}\) where \(B = P \cap R.\) Let \(f \in P= \langle \C \cup B \rangle_{\D} \colon H^{\infty}.\) Then \(H^n f \in \langle \C\cup B \rangle_{\D}\) for some \(n \in \N.\) Thus \(H^nf^n \in \langle \C\cup B \rangle_{\D}.\) Passing to the radical \(\D\)-ideal generated by \(\C \cup B,\) we have \(Hf \in \{\C \cup B\}_{\D}.\) Therefore we have that \(P \subseteq \{\C \cup B\}_{\D}\colon H.\) Consequently, \(P = \{\C\cup B\}_{\D}\colon H.\)

As we assume that \(R\) has the ascending chain condition on perfect \(\D\)-ideals, there exists a finite set \(\Phi\) such that \(\{\Phi\}_{\D} = B.\) Thus \(P = \{\C \cup \Phi \}_{\D} : H.\) As \(\C\) is a characteristic set, it is finite and so \(\C \cup \Phi\) is a finite set. As \(H \notin P,\) we are done.
\end{proof}

We now present some consequences of the above theorem.

\begin{corollary}\label{final-corollary}
    Let \(S\) be any finitely \(\D\)-generated \(\D^*\)-algebra over \((R,e).\) If \(R\) has the ascending chain condition on perfect \(\D\)-ideals, then \(S\) has the ascending chain condition on perfect \(\D\)-ideals.
\end{corollary}

\begin{proof}
    Denote the generators of \(S\) by \(\bar{a} = (a_1,\ldots,a_n).\) By \autoref{univ-prop}, there is a surjective \((R,e)\)-algebra homomorphism \(\phi \colon R\{\bar{x}\}_{\D^*} \to S\) that maps \(x_i\) to \(a_i\) for each \(i.\) We observe that every preimage of a perfect \(\D\)-ideal of \(S\) is a perfect \(\D\)-ideal of \(R\{\bar{x}\}_{\D^*}.\) Thus any chain of perfect \(\D\)-ideals of \(S\) corresponds to a chain of perfect \(\D\)-ideals of \(R\{\bar{x}\}_{\D^*}.\) As \(R\) has the ascending chain condition on perfect \(\D\)-ideals, by \autoref{d-basis-thm}, so does \(R\{\bar{x}\}_{\D^*}\) and thus this chain of perfect \(\D\)-ideals stabilises. Then the chain of perfect \(\D\)-ideals of \(S\) must also stabilise.
\end{proof}

\begin{remark}\label{consequence-differential}
    We observe how to recover the classical cases of the differential basis theorem and the difference-differential basis theorem in characteristic zero.
    \begin{itemize}
        \item Let \(\D = K[\epsilon]/(\epsilon)^2\) with the usual \(K\)-algebra structure. As in \autoref{Degs}(a), we denote the operators associated to the basis \(\{1,\epsilon\}\) by \(\sigma\) and \(\delta,\) respectively. Let \((R,e)\) be a \(\D^*\)-ring such that \(\sigma\) is the identity on \(R;\) i.e. \((R,\delta)\) is a differential ring. In \autoref{diff-poly}, we note that there is a surjective \((R,e)\)-algebra homomorphism from \(R\{x\}_{\D^*}\) to \(R\{x\}_{\delta}\) that maps \(\sigma \delta^i x\) to \(\delta^i x.\) In \autoref{differential-ideals}, we observed that perfect \(\D\)-ideals of \((R,e)\) correspond to the radical differential ideals of \((R,\delta).\) By \autoref{final-corollary}, the finitely \(\D\)-generated \(\D^*\)-algebra \(R\{x\}_{\delta}\) has the ascending chain condition on perfect \(\D\)-ideals; i.e. \(R\{x\}_{\delta}\) has the ascending chain condition on radical differential ideals.
        \item Let \(\D = K[\nu]/(\nu)^2 \times K\) with the natural \(K\)-algebra structure. As in \autoref{Degs}(c), denote by \(\sigma_0 = \delta_0,\) \(\delta = \delta_1\) and \(\sigma = \delta_2\) the operators associated to the basis \(\{\epsilon_0,\epsilon_1,\epsilon_2\}\) where \(\epsilon_0 = (1,0),\) \(\epsilon_1 = (\nu,0)\) and \(\epsilon_2 = (0,1).\) Let \((R,e)\) be a \(\D^*\)-ring such that \(\sigma_0\) is the identity on \(R;\) in other words, the structure of \((R,e)\) is that of a difference-differential ring with endomorphism \(\sigma\) and derivation \(\delta\) where the operators commute. Similarly to the differential case, there is a surjective homomorphism from \(R\{x\}_{\D^*} \to R\{x\}_{\sigma,\delta}\) that maps \(d^{(\theta_0,\theta_1,\theta_2)}x\) to \(d^{(\theta_1,\theta_2)}x\) if \(\theta_0 > 0.\) As noted in \autoref{perfect-ideal-corr}, the perfect \(\D\)-ideals of a difference ring correspond to the perfect difference ideals. Thus perfect \(\D\)-ideals of a difference-differential ring are perfect difference-differential ideals. By \autoref{final-corollary}, \(R\{x\}_{\sigma, \delta}\) has the ascending chain condition on perfect \(\D\)-ideals; i.e. \(R\{x\}_{\sigma, \delta}\) has the ascending chain condition on perfect difference-differential ideals.
    \end{itemize}
\end{remark}

We note that in addition to recovering previously established results, we can also use the \(\D^*\)-basis theorem to establish new cases not previously seen in the literature. Let \(\D = K[\epsilon]/(\epsilon)^{n+1}\) with the usual \(K\)-algebra structure. Denote by \(\delta_0, \delta_1,\ldots,\delta_n\) the operators associated to the basis \(\{1,\epsilon,\epsilon^2,\ldots,\epsilon^n\}.\) Then \((R,e)\) is a \(\D\)-ring if and only if \(\delta_0\) is a \(R\)-endomorphism and for all \(r,s \in R,\) \(\delta_i(rs) = \sum_{j+k=i}\delta_j(r)\delta_k(s).\) In the case where \(\delta_0 = \id_R,\) \((\delta_1,\ldots,\delta_n)\) yields a truncated Hasse-Schmidt derivation whose components commute (see \autoref{Degs}(d)). As this is a \(\D^*\)-ring, we can use \autoref{d-basis-thm} to establish a basis theorem for rings with commuting truncated Hasse-Schmidt derivations.

\bigskip

%\printbibliography

\bibliographystyle{plain}
\bibliography{references}

\begin{thebibliography}{10}

\bibitem{cohn_1970}
Richard~M. Cohn.
\newblock A difference-differential basis theorem.
\newblock {\em Canad. J. Math.}, 22(6):1224–1237, 1970.

\bibitem{Figueira_2011}
Diego Figueira, Santiago Figueira, Sylvain Schmitz, and Philippe Schnoebelen.
\newblock Ackermannian and primitive-recursive bounds with {D}ickson’s lemma.
\newblock In {\em 2011 IEEE 26th Annual Symposium on Logic in Computer Science}, page 269–278. IEEE, June 2011.

\bibitem{kolchin73}
E.~Kolchin.
\newblock {\em Differential algebra and algebraic groups}.
\newblock Pure and applied mathematics; v. 54. Acad. Press, New York, 1973.

\bibitem{kolchinfrench}
Ellis~R. Kolchin.
\newblock Le théorème de la base finie pour les polynômes différentiels.
\newblock {\em Séminaire Dubreil. Algèbre et théorie des nombres}, 14(1):1--15, 1961.

\bibitem{levin08}
A.~Levin.
\newblock {\em Difference Algebra}.
\newblock Algebr. Appl. Springer Dordrecht, 2008.

\bibitem{mohamed2022weil}
Shezad Mohamed.
\newblock The {W}eil descent functor in the category of algebras with free operators.
\newblock {\em J. Algebra}, 640:216--252, 2024.

\bibitem{MSGen}
Rahim Moosa and Thomas Scanlon.
\newblock Generalized {H}asse-{S}chmidt varieties and their jet spaces.
\newblock {\em Proc. Lond. Math. Soc. (3)}, 103(2):197–234, February 2011.

\bibitem{mtfo0char}
Rahim Moosa and Thomas Scanlon.
\newblock Model theory of fields with free operators in characteristic zero.
\newblock {\em J. Math. Log.}, 14, 12 2012.

\bibitem{raudenbushdifferential}
H.~W. Raudenbush.
\newblock Ideal theory and algebraic differential equations.
\newblock {\em Trans. Amer. Math. Soc.}, 36:361--368, 1934.

\bibitem{rittdifference}
J.~F. Ritt and H.~W. Raudenbush.
\newblock Ideal theory and algebraic difference equations.
\newblock {\em Trans. Amer. Math. Soc.}, 46:445--452, 1939.

\bibitem{ritt1932differential}
J.F. Ritt.
\newblock {\em Differential Equations from the Algebraic Standpoint}.
\newblock Amer. Math. Soc. Colloq. Publ. American Mathematical Society, 1932.

\bibitem{poisson}
{Susan J.} Sierra and Omar {Leon Sanchez}.
\newblock A {P}oisson basis theorem for symmetric algebras of infinite-dimensional {L}ie algebras.
\newblock {\em Ark. Mat.}, 61(2):375--412, November 2023.

\end{thebibliography}

\end{document}